\documentclass{amsart}

\usepackage{amsmath, amsthm, amssymb, amstext, amsfonts}
\usepackage{enumerate}
\usepackage{graphicx}
\usepackage[T1]{fontenc} 
\usepackage[applemac]{inputenc}

\theoremstyle{plain}

\newtheorem{thm}{Theorem}[section]
\newtheorem{lem}[thm]{Lemma}
\newtheorem{cor}[thm]{Corollary}
\newtheorem{prop}[thm]{Proposition}

\newtheorem{prob}[thm]{Problem}
\newtheorem*{prob*}{Problem}

\theoremstyle{definition}

\newtheorem{ex}[thm]{Example}

\newcommand{\R}{\ensuremath{\mathbb{R}}}
\newcommand{\N}{\ensuremath{\mathbb{N}}}

\newcommand{\cB}{\ensuremath{\mathcal{B}}}
\newcommand{\cO}{\ensuremath{\mathcal{O}}}

\newcommand{\inv}{\ensuremath{^{-1}}}

\newcommand{\sub}{\subseteq}

\def\qi{quasi-iso\-metric}
\def\qiy{quasi-iso\-me\-try}
\def\qis{quasi-iso\-me\-tries}

\def\qg{quasi-geo\-desic}

\newcommand{\comment}[1]{}

\newcommand{\nat}{{\mathbb N}}
\newcommand{\real}{{\mathbb R}}

\newcommand{\lrarrow}{{\leftrightarrow}}

\def\hm{semimetric}
\def\Hm{Semimetric}

\def\?#1{\vadjust{\vbox to 0pt{\vss\vskip-8pt\leftline{%
     \llap{\hbox{\vbox{\pretolerance=-1
     \doublehyphendemerits=0\finalhyphendemerits=0
     \hsize16truemm\tolerance=10000\small
     \lineskip=0pt\lineskiplimit=0pt
     \rightskip=0pt plus16truemm\baselineskip8pt\noindent
     \hskip0pt        
     #1\endgraf}\hskip7truemm}}}\vss}}}

\newenvironment{txteq}
  {
    \begin{equation}
    \begin{minipage}[c]{0.85\textwidth} 
    \em                                
  }
  {\end{minipage}\end{equation}\ignorespacesafterend}
\newenvironment{txteq*}
  {
    \begin{equation*}
    \begin{minipage}[c]{0.9\textwidth} 
    \em                                
  }
  {\end{minipage}\end{equation*}\ignorespacesafterend}
\newenvironment{txteqtag}[1]
  {
    \begin{equation}\tag{#1}
    \begin{minipage}[c]{0.85\textwidth} 
    \em                                
  }
  {\end{minipage}\end{equation}\ignorespacesafterend}

\begin{document}

\title[Quasi-isometry invariance of hyperbolicity]{Quasi-isometry invariance of hyperbolicity in semimetric spaces, digraphs and semigroups}
\author{Matthias Hamann}
\thanks{Funded by the German Research Foundation (DFG) -- project number 448831303.}
\address{Matthias Hamann, Mathematics Institute, University of Warwick, Coventry, UK}
\date{}

\begin{abstract}
Gray and Kambites introduced a notion of hyperbolicity in the setting of \hm\ spaces like digraphs or semigroups.
We will prove that under a small additional geometric assumption their notion of hyperbolicity is preserved by \qis.
Applied to semigroups, this will partially solve a problem of Gray and Kambites.
\end{abstract}


\maketitle

\section{Introduction}\label{sec_Intro}

Since Gromov's paper~\cite{gromov} on hyperbolic groups, hyperbolicity is a topic that received a lot of interest, mostly for geodesic metric spaces, graphs or groups, see e.\,g.\ \cite{ABCFLMSS, BridsonHaefliger,CoornDelPapa,GhHaSur, KB-BoundaryHypGroup}.
Occasionally, some of these assumptions are dropped.
One such example is the paper of V\"ais\"al\"a~\cite{V-GromovHyp}, where the metric spaces under consideration need not be geodesic.
In the present paper, we try to extend the perhaps most important result, \qiy\ invariance, to a generalisation of hyperbolicity in spaces, where the distance function need not be symmetric.
The particular notion we are interested in was defined by Gray and Kambites~\cite{GK-HyperbolicDigraphsMonoids} in order to have a geometric notion of hyperbolicity for monoids.

Gray and Kambites were not the first to look for a notion of hyperbolicity for monoids but their notion is the first geometric one that takes directions into account.
Previous definitions either were based on context-free rewriting systems, see e.\,g.\ \cite{C-HyperbolicMonoids,CM-RewritingSystemsHyperbolicity,DG-WordHyperbolicSemigroups,FK-HyperbolicGroupsSemigroups,HT-HyperbolicMonoids}, or looked at the underlying undirected graph of Cayley digraphs of semigroups, see e.\,g.\ \cite{C-HyperbolicMonoids,CS-InfiniteWordsRewritingSystems,FK-HyperbolicGroupsSemigroups}.
Portilla et al.\ \cite{PRST-HyperbolicDigraphs} considered a notion of hyperbolicity in (a certain class of) digraphs that is also defined via hyperbolicity of their underlying undirected graphs.
The advantage of a geometric notion for monoids that takes directions into account is that it also applies to other objects such as directed graphs, digraphs for short, or geodesic \hm\ spaces.
The latter are spaces where the distance function has $[0,\infty]$ as its range and we do ask for all condition needed for a metric except for $d(x,y)=d(y,x)$ for distinct $x$ and~$y$, see Section~\ref{sec_HMSpaces} for the formal definition.
In the literature, these spaces are also known as quasi-metric or asymmetric metric spaces, but we use the term "semimetric" here.
We refer to Gray and Kambites \cite[Section~2]{GK-SemimetricSpaces} for a discussion about these terms.

In Section ~\ref{sec_ex}, we will see that the notion of Gray and Kambites indeed generalizes the one for geodesic metric spaces.
Their notion is based on a thinness condition for geodesic triangles, see Section~\ref{sec_thin} for the precise definitions.

While Gray and Kambites were mostly interested in properties of hyperbolic monoids, such as the word problem or other semigroup-theoretic decision problems and the question of finite presentability of those monoids, we focus on a geometric aspect of the spaces themselves.
We consider the question whether this notion of hyperbolicity is preserved by \qis.

In the situation of metric spaces, in order to prove that \qis\ preserve hyperbolicity, first, it is shown that hyperbolicity is equivalent with geodesic stability and the latter is easily seen to be preserved by \qis.
In order to prove that equivalence, one uses that both properties are also equivalent to exponential divergence of geodesics.

In the situation of \hm\ spaces, we do no longer have equivalence between these three properties.
More precisely, we will prove with a small additional geometric assumption that hyperbolicity implies both other properties but neither of the other two implies hyperbolicity nor do they imply each other, see Sections~\ref{sec_DivGeo} and~\ref{sec_GeoStab}.
Still, this is enough to prove that hyperbolicity (and geodesic stability) is preserved by \qis, see Section~\ref{sec_QI}.
It stays an open problem whether (exponential) divergence of geodesics is preserved by \qis.

Let us note that the definitions of e.\,g.\ \qis\ and geodesics are slightly different in \hm\ spaces compared to metric spaces.
For the introduction and as intuition, it is fine to just think of these notions as being similar to the ones in metric spaces and that they carry over the important properties that we will need.
We omit these definitions here in the introduction and instead refer to the later sections for these definitions.

Let us briefly talk about the geometric assumption on the semimetric spaces that we need.
In geodesic metric spaces $X$, we find for every point $x\in X$ and points $y,z$ in the ball of radius~$r$ around~$x$ a $y$-$z$ geodesic of length at most~$2r$.
For \hm\ spaces~$Y$, this is no longer true: if $x\in Y$ and $y,z\in Y$ with $d(x,y)<r$ and $d(x,z)<r$, then $d(y,z)$ need not be bounded or even be finite; it may also happen that the out-ball of radius~$r$ around~$x$ contains an infinite geodesic.
The same also holds for the in-balls.
We will ask our \hm\ spaces to satisfy a condition that bounds the lengths of geodesics in out-balls or in-balls of finite radius in terms of the radius.
So our results still cover a large proper subclass of the geodesic \hm\ spaces that contain all geodesic metric spaces.
For example, we will see that all digraphs of bounded degree satisfy this condition.
As a consequence, when applying our results to semigroups, we only obtain results for right cancellative semigroups.

In the final section, Section~\ref{sec_semigroups}, we apply our results to semigroups.
Gray and Kambites~\cite{GK-HyperbolicDigraphsMonoids} defined hyperbolicity for finitely generated semigroups by asking one of the Cayley digraphs be hyperbolic for a finite generating set.
They ask whether for a finitely generated hyperbolic semigroup each of its Cayley digraphs with respect to finite generating sets is hyperbolic.
As corollary from our result on \qiy\ invariance of hyperbolicity, we answer this positively for right cancellative semigroups.

\section{Preliminaries}

In this section, we define all basic notions related to pseudo-\hm\ spaces and to digraphs.

\subsection{\Hm\ spaces}\label{sec_HMSpaces}

Let $X$ be a set.
A map $d\colon X\times X\to[0,\infty]$ is a \emph{pseudo-\hm} if
\begin{enumerate}[(i)]
\item $d(x,x)=0$ for all $x\in X$ and
\item $d(x,y)\leq d(x,z)+d(z,y)$ for all $x,y,z\in X$.
\end{enumerate}
We then call $(X,d)$ a \emph{pseudo-\hm\ space}.

A pseudo-\hm\ is a \emph{\hm} if $d(x,y)=0$ if and only if $x=y$ for all $x,y\in X$.
Then we call $(X,d)$ a \emph{\hm\ space}.

For $x,y$ in a (pseudo-)\hm\ space $X$, we set
\[
d^\lrarrow(x,y):=\min \{d(x,y),d(y,x)\}.
\]

(Pseudo-)\Hm\ spaces $X$ come along with two natural topologies that are defined via the out-balls and the in-balls.
For $r\geq 0$ and $x\in X$, we set the \emph{out-ball} and the \emph{open out-ball} of radius~$r$ around~$x$ as
\[
\cB^+_r(x):=\{y\in X\mid d(x,y)\leq r\}, \qquad \mathring{\cB}^+_r(x):=\{y\in X\mid d(x,y)< r\}
\]
and the \emph{in-ball} and the \emph{open in-ball} of radius~$r$ around~$x$ as
\[
\cB^-_r(x):=\{y\in X\mid d(y,x)\leq r\}, \qquad \mathring{\cB}^-_r(x):=\{y\in X\mid d(y,x)< r\},
\]
respectively.

The open out-balls $\mathring{\cB}^+_r(x)$ for all $r\geq 0$ and $x\in X$ generate the \emph{forward} topology $\cO_f$ and the open in-balls $\mathring{\cB}^-_r(x)$ for all $r\geq 0$ and $x\in X$ generate the \emph{backward} topology $\cO_b$.

For $a,b\in\R$ with $a<b$, we consider the interval $[a,b]$ with the above two topologies.
A function $P\colon [a,b]\to X$ is a \emph{directed path} if it is continuous with respect to the forward topologies of $[a,b]$ and~$X$ and if it is continuous with respect to the backward topologies of $[a,b]$ and~$X$.
Then $P(a)$ is the \emph{starting point} and $P(b)$ is the \emph{end point} of~$P$.
The \emph{length} $\ell(P)$ of~$P$ is defined as
\[
\ell(P):=\lim_{N\to\infty}\sum_{i=1}^N d(P(t_{i-1}),P(t_i))
\]
with $t_i:=a+i(b-a)/N$ for all $0\leq i\leq N$.
For $x,y\in X$, a \emph{directed $x$-$y$ path} is a directed path with starting point~$x$ and end point~$y$.
For $U,V\sub X$, a \emph{directed $U$-$V$ path} is a directed path with starting point in~$U$ and end point in~$V$.

Let $P\colon [a,b]\to X$ and $Q\colon [a',b']\to X$ be directed paths.
They are \emph{parallel} if $P(a)=Q(a')$ and $P(b)=Q(b')$ and they are \emph{composable} if $P(b)=Q(a')$.
If they are composable, then we denote by $PQ$ the resulting directed path, their \emph{composition}.

We say that $x\in X$ \emph{lies on $P$} if it is in the image of~$P$.
For $x,y$ on~$P$ with $x=P(r)$ and $y=P(R)$ such that $r<R$, we denote by $xPy$ the subpath of~$P$ between $x$ and~$y$.
With a slight abuse of notation, we denote for $r\geq 0$ by $\cB^+_r(P)$, by $\cB^-_r(P)$, the out-ball, the in-ball, of radius $r$ around the image of~$P$, respectively.

Let $(X,d_X)$ and $(Y,d_Y)$ be \hm\ spaces.
A map $f\colon X\to Y$ is a \emph{\qi\ embedding} if there are constants $\gamma\geq 1$ and $c\geq 0$ such that
\[
\gamma\inv d_X(x,x')-c\leq d_Y(f(x),f(x'))\leq\gamma d_X(x,x')+c
\]
for all $x,x'\in X$.
It is a \emph{\qiy} if additionally for every $x\in X$ there is $y\in Y$ such that $d(f(x),y)\leq c$ and $d(y,f(x))\leq c$.
Then we say that $X$ is \emph{\qi} to~$Y$.
If we want to emphasize on the particular constants $\gamma$ and~$c$, we talk about \emph{$(\gamma,c)$-\qis}.
By Gray and Kambites~\cite[Proposition 1]{GK-SemimetricSpaces}, being \qi\ is an equivalence relation.
An \emph{isometry} is a $(1,0)$-\qiy.

A \emph{geodesic} from $x\in X$ to $y\in X$, also called an \emph{$x$-$y$ geodesic}, is a directed path $P$ from~$x$ to~$y$ with $d(u,v)=\ell(uPv)$ for all $u,v$ on~$P$ with $v$ on $uPy$.
For subsets $U,V\sub X$, a \emph{$U$-$V$ geodesic} is a directed $U$-$V$ path that is a geodesic.
For $\gamma\geq 1$ and $c\geq 0$, a \emph{$(\gamma,c)$-\qg} from~$x$ to~$y$ is a directed path from~$x$ to~$y$ with
\[
\ell(uPv)\leq \gamma d(u,v)+c
\]
for all $u,v$ on~$P$ with $v$ on~$uPy$.
We call $X$ \emph{geodesic} if there exists an $x$-$y$ geodesic for all $x,y\in X$ with $d(x,y)<\infty$.

Note that $x$-$y$ geodesics are not isometric images of $[0,d(x,y)]$ since we make no restrictions on $d(P(b),P(a))$ for $0\leq a\leq b\leq 1$.
Similarly, \qg s are not \qi\ images of $[0,d(x,y)]$.

\subsection{Digraphs}\label{sec_digraphs}

Here, a digraph $D=(V(D),E(D))$ is an oriented multigraph.
In particular, we allow our digraphs to have loops and multiple edges between the same two vertices.
For a subset $X$ of~$V(D)$, we denote by $D[X]$ the digraph \emph{induced by} $X$, i.\,e.\ the digraph with vertex set $X$ and all edges of~$D$ both of whose incident vertices lie in~$X$.

A \emph{directed path} is a sequence $x_0\ldots x_n$ of vertices such that $x_ix_{i+1}\in E(D)$ for all $0\leq i<n$.
A \emph{proper directed path} $P$ is a sequence $x_0\ldots x_n$ of pairwise distinct vertices such that $x_ix_{i+1}\in E(D)$ for all $0\leq i<n$.\footnote{With our terminology, we differ from the usual one for digraphs: what we call directed paths are usually directed walks and our proper directed paths are directed paths. We chose this different terminology because now directed paths in digraphs are the same regardless whether we consider it as a digraph or as a \hm\ space.}
The \emph{length} $\ell(P)$ of~$P$ is $n$, the number of edges of the path.

If $x_0,x_1,\ldots$ are distinct vertices in~$D$ with $x_ix_{i+1}\in E(D)$ we call $x_0x_1\ldots$ a \emph{ray}.
If we have $x_{i+1}x_i\in E(D)$ instead we say that $x_0x_1\ldots$ is an \emph{anti-ray}.

For $x,y\in V(D)$, let $d(x,y)$ be the length of a shortest directed path from~$x$ to~$y$ or $\infty$ if no such path exists.
Note that in contrast to graphs, this distance function is not a metric but a \hm.
The \emph{out-degree} of a vertex $x\in V(D)$ is the number of vertices $y\in V(D)$ with $d(x,y)=1$ and the \emph{in-degree} of~$x$ is the number of vertices $y\in V(D)$ with $d(y,x)=1$.
A digraph is \emph{locally finite} if its in- and out-degrees are all finite.

We consider the the following \hm\ space associated to a digraph $D$:
we consider the $1$-complex of the underlying undirected (multi-)graph of~$D$.
The distance between two points $x,y$ corresponding to vertices $u,v$, respectively, is set as follows: $d(x,y):=d(u,v)$.
For an inner point $x$ of an edge $uv$, its distance to a point $y$ corresponding to a vertex $w$ or an inner point of an edge $ww'$ is set as $d(x,y):=d(x,v)+d(v,w)+d(w,y)$, where we consider the edges $e$ to be directed paths of length~$1$ in the \hm\ space.

\section{Thin triangles}\label{sec_thin}

Let $X$ be a geodesic \hm\ space.
A \emph{triangle} consists of three points of~$X$ and three directed paths, one between every two of those points.
We call these paths the \emph{sides} and the three point the \emph{end points} of the triangle.
The triangle is \emph{geodesic} if all three sides are geodesics and it is \emph{transitive} if two of its sides are composable and the resulting directed path is parallel to the third side.

We consider thin triangles as defined by Gray and Kambites in~\cite{GK-HyperbolicDigraphsMonoids}.
Let $\delta\geq 0$.
A geodesic triangle is \emph{$\delta$-thin} if the following holds:
\begin{txteq*}
if $P,Q,R$ are the sides of the triangle and the starting point of~$P$ is either the starting or the end point of~$Q$ and the last point of~$P$ is either the starting or the end point of~$R$, then $P$ is contained in $\cB^+_\delta(Q)\cup \cB^-_\delta(R)$.
\end{txteq*}
If all geodesic triangles in~$X$ are $\delta$-thin then $X$ is \emph{$\delta$-hyperbolic}.
We call $X$ \emph{hyperbolic} if it is $\delta'$-hyperbolic for some $\delta'\geq 0$.

We note that this definition differs slightly from the definition of Gray and Kambites in that they only ask transitive geodesic triangles to be $\delta$-thin.
However, our first result is that -- up to the constant $\delta$ -- both definitions are equivalent.

\begin{prop}\label{prop_transTriSuffices}
Let $X$ be a geodesic \hm\ space and $\delta\geq 0$.
If all transitive geodesic triangles are $\delta$-thin, then all geodesic triangles are $3\delta$-thin.
\end{prop}

\begin{proof}
Let $x,y\in X$ such that there are geodesics $P$ and~$Q$ from~$x$ to~$y$ and from~$y$ to~$x$, respectively.
Let $R$ be the trivial directed path with image~$x$.
Then $P,Q,R$ form the sides of a transitive geodesic triangle.
So $P$ lies in $\cB^+_\delta(R)\cup \cB^-_\delta(Q)$.
Since $P$ is a geodesic, only its directed subpath with starting point $x$ and end point $P(\delta)$ is contained in $\cB^+_\delta(R)$.
The remaining part of~$P$ lies in $\cB^-_\delta(Q)$ and thus all of $P$ lies in $\cB^-_{2\delta}(Q)$.
Similarly, if we take the third side as trivial directed path~$y$, we obtain that $P$ lies in $\cB^+_{2\delta}(Q)$.

Now let us consider an arbitrary geodesic triangle that is not transitive and let $P,Q,R$ be its sides and $x,y,z$ be its end points such that $P$ is an $x$-$y$ geodesic, $Q$ is a $y$-$z$ geodesic and $R$ is a $z$-$x$ geodesic.
Then $P$ and~$Q$ are composable and hence there is an $x$-$z$ geodesic~$S$.
By our previous argumentation, we know that $R$ lies in $\cB^-_{2\delta}(S)\cap\cB^+_{2\delta}(S)$.
Since the geodesic triangle with sides $P,Q,S$ is $\delta$-thin, we obtain that $R$ lies in ${\cB^-_{3\delta}(P)\cap\cB^+_{3\delta}(Q)}$.
\end{proof}

Contrary to metric spaces, in a \hm\ space $X$, the lengths of geodesics with starting and end point in $\cB^+_r(x)$ for $x\in X$ and $r\in R$ need not be bounded.
(In metric spaces, this is bounded by~$2r$.)
However, we will often restrict ourselves to situations, where this is satisfied.
We define the following two properties.

\begin{txteqtag}{B1}\label{itm_Bounded1}
There exists a function $f\colon \R\to\R$ such that for every $x\in X$, for every $r\geq 0$ and for all $y,z\in \cB^+_r(x)$ the distance $d(y,z)$ is either $\infty$ or bounded by $f(r)$.
\end{txteqtag}
\begin{txteqtag}{B2}\label{itm_Bounded2}
There exists a function $f\colon \R\to\R$ such that for every $x\in X$, for every $r\geq 0$ and for all $y,z\in \cB^-_r(x)$ the distance $d(y,z)$ is either $\infty$ or bounded by $f(r)$.
\end{txteqtag}

These properties are satisfied in several applications as we will see now.
As mentioned above, if $X$ is also a metric space, then it satisfies (\ref{itm_Bounded1}) and (\ref{itm_Bounded2}) for the function $f(r)=2r$.

\begin{lem}\label{lem_bddDegImpliesB1B2}
Every hyperbolic digraph of bounded in- and bounded out-degree satisfies (\ref{itm_Bounded1}) and~(\ref{itm_Bounded2}).
\end{lem}

\begin{proof}
Let $D$ be a $\delta$-hyperbolic digraph of in- and out-degree at most~$\rho$ and let $x\in V(D)$.
Then the number of vertices in $\cB^+_k(x)$ and the number of vertices in $\cB^-_k(x)$ is bounded by $f_k:=\sum_{i=0}^k(\rho-1)^i$.
Let $P$ be a geodesic with starting and end vertex in $\cB^+_k(x)$.
Let $Q$ be a geodesic from~$x$ to the starting vertex of~$P$ and let $R$ be a geodesic from~$x$ to the end vertex of~$P$.
So $P,Q,R$ form the sides of a geodesic triangles.
Since this is $\delta$-thin, each vertex of~$P$ lies either in a ball $\cB^+_\delta(y)$ for some vertex $y$ on~$Q$ or in $\cB^-_\delta(z)$ for some vertex $z$ on~$R$.
In these balls, there are at most $k f_\delta+kf_\delta$ many vertices.
Thus, this number is also an upper bound on the length of~$P$.
\end{proof}

It follows from Lemma~\ref{lem_bddDegImpliesB1B2} that locally finite transitive\footnote{Recall that a digraph is \emph{transitive} if for all $x,y\in V(D)$ there is an automorphism $\alpha$ of~$D$ that maps $x$ to~$y$.} hyperbolic digraphs satisfy (\ref{itm_Bounded1}) and (\ref{itm_Bounded2}).
Another class of digraphs satisfying (\ref{itm_Bounded1}) and (\ref{itm_Bounded2}) are Cayley digraphs of finitely generated right cancellative hyperbolic semigroups, see Section~\ref{sec_semigroups} for more details.

A major property that follows from (\ref{itm_Bounded1}) and (\ref{itm_Bounded2}) is formulated in our next propositions.
Note that, by Lemma~\ref{lem_bddDegImpliesB1B2}, that proposition is a generalization of a result of Gray and Kambites \cite[Lemma 3.1]{GK-HyperbolicDigraphsMonoids}.

\begin{prop}\label{prop_reverseDistanceShortDelta1}
Let $\delta\geq 0$ and let $X$ be a $\delta$-hyperbolic geodesic \hm\ space that satisfies (\ref{itm_Bounded1}) for the function $f\colon \R\to\R$ and (\ref{itm_Bounded2}) for the function~$g\colon \R\to\R$.
\begin{enumerate}[\rm (i)]
\item\label{itm_reverseDistanceShortDelta1_1} If $P,Q,R$ are the sides of a geodesic triangle such that the starting point of~$P$ is either the starting or the end point of~$Q$ and the end point of~$P$ is either the starting or the end point of~$R$, then we have
\[
\ell(P)\leq (\ell(Q)/\varepsilon)f(\delta+\varepsilon)+(\ell(R)/\varepsilon)g(\delta+\varepsilon)
\]
for all $\varepsilon>0$.
\item\label{itm_reverseDistanceShortDelta1_2} If $x,y \in X$ with $d(x,y)\neq\infty$ and $d(y,x)\neq\infty$, then we have
\[
d(x,y)\leq (d(y,x)/\varepsilon) f(\delta+\varepsilon) + g(\delta)
\]
and
\[
d(x,y)\leq (d(y,x)/\varepsilon) g(\delta+\varepsilon)+f(\delta)
\]
for all $\varepsilon>0$.
\end{enumerate}
\end{prop}

\begin{proof}
Let $\varepsilon>0$.
By assumption, $P$ lies in $\cB^+_\delta(Q)\cup\cB^-_\delta(R)$.
Let $a$ be a point on~$Q$ and let $u=P(r)$, $v=P(R)$  in $\cB^+_{\delta+\varepsilon}(a)$ such that $r<R$.
By (\ref{itm_Bounded1}), we have $d(u,v)\leq f(\delta+\varepsilon)$.
All points on~$P$ that lie in $\cB^+_\delta(Q)$ lie in the union of all $\cB^+_{\delta+\varepsilon}(Q(i\varepsilon))$ with $i\in\N$ and $i\varepsilon\leq \ell(Q)$.
For each $i\leq \ell(Q)/\varepsilon$, let $P_i$ be a smallest subpath of~$P$ containing all points of~$P$ that lie in $\cB^+_{\delta+\varepsilon}(Q(i\varepsilon))$.
The lengths of these paths sum up to at most $(\ell(Q)/\varepsilon)f(\delta+\varepsilon)$.
Analogously, the lengths of the smallest subpaths of~$P$ that contain all point on~$P$ that lie in $\cB^-_{\delta+\varepsilon}(R(i\varepsilon))$ sum up to at most $(\ell(R)/\varepsilon)g(\delta+\varepsilon)$.
This proves~(\ref{itm_reverseDistanceShortDelta1_1}).

Finally, (\ref{itm_reverseDistanceShortDelta1_2}) follows directly from (\ref{itm_reverseDistanceShortDelta1_1}): we just choose $P$ to be an $x$-$y$ geodesic and in the first case $R$ to be the trivial path $y$ and $Q$ a $y$-$x$ geodesic and in the second case $R$ to be a $y$-$x$ geodesic and $Q$ the trivial path~$x$.
Note that it suffices to take $g(\delta)$ or $f(\delta)$ for the trivial path.
\end{proof}

In our last result in this section, we prove that we can control that side in a transitive triangle that is parallel to the composition of the other two even further than the definition of a $\delta$-thin triangle indicates.

\begin{lem}\label{lem_parallelSideCloseToAndFrom}
Let $\delta\geq0$ and let $X$ be a $\delta$-hyperbolic geodesic \hm\ space that satisfies (\ref{itm_Bounded1}) and (\ref{itm_Bounded2}) for the function $f\colon\R\to\R$.
Let $P,Q,R$ be the sides of a geodesic triangle such that $P$ and~$Q$ are composable and their composition is parallel to~$R$.
Then $R$ lies in the out-ball of radius $6\delta+2\delta f(\delta+1)$ around $P\cup Q$ and in the in-ball of the same radius around $P\cup Q$.
\end{lem}

\begin{proof}
Set $k^*:=4\delta+2\delta f(\delta+1)$ and $k:=k^*+2\delta$.
Let $a$ be the starting and~$b$ be the end point of~$R$.
Let $x$ be a point on~$R$.
Then $x\in\cB^+_\delta(P)\cup\cB^-_\delta(Q)$.
If $x\in\cB^+_\delta(P)$, then $d(P,x)\leq \delta\leq k^*\leq k$.
So we may assume that there is a point~$y$ on~$Q$ with $d(x,y)\leq\delta$.
Then $y\in\cB^+_\delta(P)\cup\cB^-_\delta(R)$.

Let us first assume that there is a point $z$ on~$P$ with $d(z,y)\leq\delta$.
Let $P'$ be a $z$-$y$ geodesic, $Q'$ an $a$-$y$ geodesic and $R'$ an $x$-$y$ geodesic.
Then every point of~$Q'$ that lies in $\cB^-_\delta(P')$ lies in $\cB^-_{2\delta}(y)$.
Thus, $Q'$ lies completely in $\cB^+_{3\delta}(P)$.
Using Proposition~\ref{prop_reverseDistanceShortDelta1} with $\varepsilon=1$, at most the subpath of length $2\delta f(\delta+1)$ with end vertex~$x$ of $aRx$ lies in $\cB^-_\delta(R')$.
Thus, there exists $u$ on~$Q'$ with $d(u,x)\leq \delta+2\delta f(\delta+1)$ and hence we have
\[
d(P,x)\leq d(P,u)+d(u,x)\leq 4\delta+2\delta f(\delta+1)=k^*.
\]

Let us now assume that there is a point $z$ on~$R$ with $d(y,z)\leq\delta$.
If $z$ lies on $aRx$, then Proposition~\ref{prop_reverseDistanceShortDelta1} implies $d(z,x)\leq 2\delta f(\delta+1)$.
So we have
\[
d(y,x)\leq\delta+2\delta f(\delta+1)\leq k^*.
\]
Hence, we assume that $z$ lies on $xRb$.
If $d(a,x)\leq 2\delta$, we immediately obtain $d(P,x)\leq 2\delta\leq k$
Otherwise, let $x'$ be on~$aRx$ with $d(x',x)=2\delta$.
By an analogous situation for~$x'$ as for~$x$ either we have $d(P\cup Q,x')\leq k^*$ or there exists a point $v$ on $x'Rb$ with $d(Q,v)\leq\delta$ and $d(x',v)\leq 2\delta$.
Thus, we obtain $d(P\cup Q,x)\leq k$.

By a symmetric argument, we also obtain that $d(x,P\cup Q)\leq k$.
\end{proof}

\section{Examples of hyperbolic \hm\ spaces}\label{sec_ex}

We recall that for geodesic metric spaces $X$, the definition of hyperbolicity is as follows: $X$ is \emph{hyperbolic} if there exists $\delta\geq 0$ such that for all $x,y,z\in X$ and all $x$-$y$, $y$-$z$ and $x$-$z$ geodesics each of these geodesics lies in the $\delta$-neighbourhood of the other two geodesics.
It directly follows from the definitions that hyperbolic geodesic metric spaces are also hyperbolic when considering the space as a \hm\ space and using the definition from Section~\ref{sec_thin}.
This is similar to the observation of Gray and Kambites \cite[Section 2]{GK-HyperbolicDigraphsMonoids} that starting with a graph and replacing each edge by two directed edges that are oppositely oriented leads to a digraph that is hyperbolic if and only if the graph is hyperbolic when viewed as a metric space.

\medskip

Two points $x,y$ in a \hm\ space $X$ are \emph{equivalent}, if $d(x,y)<\infty$ and $d(y,x)<\infty$.
This is an equivalence relation whose classes are the strongly connected components of~$X$.
Gray and Kambites \cite[Proposition 2.5]{GK-HyperbolicDigraphsMonoids} showed that in hyperbolic digraphs, the strongly connected components are hyperbolic, too.
Their proof immediately carries over to \hm\ spaces, so we obtain the following.

\begin{prop}
Let $X$ be a $\delta$-hyperbolic geodesic \hm\ space.
Then every strongly connected component of~$X$ is $\delta$-hyperbolic.\qed
\end{prop}

\medskip

In the case of graphs, the $0$-hyperbolic graphs are the trees.
For digraphs, this is no longer the case: while Gray and Kambites \cite[Proposition 2.4]{GK-HyperbolicDigraphsMonoids} noted that digraphs whose underlying undirected graphs are trees are $0$-hyperbolic, there are more $0$-hyperbolic digraphs.
We will characterise the $0$-hyperbolic digraphs by using one property of trees: they have uniquely determined paths between every two vertices.

\begin{prop}
Let $D$ be a digraph.
Then $D$ is $0$-hyperbolic if and only if there are no $x,y\in V(D)$ with two distinct directed $x$-$y$ paths (which may be trivial).
\end{prop}

\begin{proof}
First, let us assume that $D$ is $0$-hyperbolic.
Let us suppose that there are $x,y\in V(D)$ with two distinct directed $x$-$y$ paths $P_1$ and~$P_2$.
First, let us assume that $x=y$.
Then we may assume that $P_1$ is trivial and that $P_2$ contains a vertex distinct from~$x$.
Let $z$ be an out-neighbour of~$x$ on~$P_2$ and let $a$ be a point on an edge~$e=xz$ distinct from~$x$ and~$z$.
Let $P$ be a $z$-$x$ geodesic.
Then $P$ does not contain the edge~$e$.
In particular, $a$ lies neither on~$P$ nor on the trivial directed path~$x$.
This is a contradiction to $0$-hyperbolicity.

So let us now assume that $x\neq y$ and that all vertices of~$P_1$ are distinct and all vertices of~$P_2$ are distinct.
Let $P$ be an $x$-$y$ geodesic.
We may assume that $P_1\neq P$.
Let $e=uv$ be an edge on~$P_1$ that does not lie on~$P$.
Let $Q_u^1$ and $Q_v^1$ be $x$-$u$ and $x$-$v$ geodesics, respectively, and let $Q_u^2$ and $Q_v^2$ be $u$-$y$ and $v$-$y$ geodesics, respectively.
If $e$ does not lie on $Q_v^1$, then the geodesic triangle with end vertices $x,u,v$ and sides $Q_u^1$, $Q_v^1$ and~$e$ contradicts $0$-hyperbolicity.
Thus, $e$ lies on $Q_v^1$ and the geodesic triangle with end vertices $x,v,y$ and sides $Q_v^1$, $Q_v^2$ and~$P$ contradicts $0$-hyperbolicity.
Thus, there exists a uniquely determined directed $x$-$y$ path.

Let us now assume that for all $x,y\in V(D)$ there exists a unique directed $x$-$y$ path.
If a geodesic triangle is not transitive, then there are two distinct directed $z$-$z$ paths for every end vertex of the triangle: one trivial one and the other one following all three sides of the triangle.
Thus, all geodesic triangles are transitive.
But then the composition of the composable sides must coincide with the third side and thus the triangle is $0$-thin.
\end{proof}

\section{Divergence of geodesics}\label{sec_DivGeo}

For geodesic metric spaces, hyperbolicity is equivalent to divergence of geodesics.
For \hm\ spaces, the analogue is not true.
But we will see that at least one direction holds under the assumptions of~(\ref{itm_Bounded1}) and~(\ref{itm_Bounded2}): hyperbolicity implies (exponential) divergence of geodesics.

Let us start by giving the definition of divergence of geodesics in the context of \hm\ spaces.
Let $X$ be a geodesic \hm\ space.
Let $P$ be an $x$-$y$ geodesic in~$X$ and let $0\leq R\leq d(x,y)$.
We denote by $P^x(R)$ be the point $u$ with $d(x,u)=R$ and we denote by $P^y(R)$ be the point $v$ with $d(v,y)=R$.

A function $e\colon \R\to\real$ is a \emph{divergence function} if for all $x\in X$, all geodesics $P_1,P_2$ that start or end at~$x$ and all $r, R\in\R$ the following holds: if $d(P_1^x(R),P_2)>e(0)$, then every directed $P_1$-$P_2$ path that lies outside of $\cB^+_{R+r}(x)\cup\cB^-_{R+r}(x)$ has length more than~$e(r)$.
We say that geodesics \emph{diverge (exponentially)} in~$X$ if there exists an (exponential) divergence function.
They diverge \emph{properly} if $e(r)\to\infty$ for $r\to\infty$.

\begin{prop}\label{prop_exponentialDivergence}
Let $X$ be a hyperbolic geodesic \hm\ space that satisfies (\ref{itm_Bounded1}) and~(\ref{itm_Bounded2}).
Then the geodesics diverge exponentially in~$X$.
\end{prop}

\begin{proof}
Let $\delta\geq 0$ such that $X$ is $\delta$-hyperbolic.
Let $x\in X$, let $P_1,P_2$ be two geodesics that have $x$ as their starting or end point and let $f\colon\R\to\R$ be a function such that $X$ satisfies (\ref{itm_Bounded1}) and~(\ref{itm_Bounded2}) for~$f$.
Let $\varepsilon >0$.
Set
\[
e(0):=(2\delta+\varepsilon+1) f(\delta+1)+ f(\delta)+\delta
\]
and, for $r>0$ and $k:=6\delta+2\delta f(\delta+1)$, set
\[
e(r):=2^{\frac{r-2\delta-1}{k}}-1.
\]
Let $R,r\in\R$ with $d(P_1^x(R),P_2)>e(0)$.
Let $P$ be a directed path from $P_1$ to~$P_2$ that lies completely outside of $\cB^+_{R+r}(x)\cup\cB^-_{R+r}(x)$ and let $Q$ be a geodesic from the starting point~$a$ of~$P$ to its end point~$b$.

We are going to show the following.
\begin{enumerate}[(i)]
\item\label{itm_exponentialDivergence1} If $x$ is the starting point of~$P_1$, then $d(x,Q)\leq R+2\delta$.
\item\label{itm_exponentialDivergence2} If $x$ is the end point of~$P_1$, then $d(Q,x)\leq R+\delta$.
\end{enumerate}

If $P_1$ ends at~$x$, then there is a point $v$ either on $Q$ with $d(v,P_1^x(R))\leq\delta$ or on~$P_2$ with $d(P_1^x(R),v)\leq\delta$.
By the choice of~$R$, we know that $v$ must lie on~$Q$ and this shows~(\ref{itm_exponentialDivergence2}).
So let $P_1$ start at~$x$ and let $v$ be a point either on~$Q$ with $d(P_1^x(R),v)\leq\delta$ or on~$P_2$ with $d(v,P_1^x(R))\leq\delta$.
If $v$ lies on~$Q$, then we have~(\ref{itm_exponentialDivergence1}).
So let us assume that $v$ lies on~$P_2$.
If $x$ is the end point of~$P_2$, then the directed path $P_1^x(R)P_1aQbP_2v$ shows that we find a $P_1^x(R)$-$v$ geodesic, whose length is at most $\delta f(\delta+1)+f(\delta)$ by Proposition~\ref{prop_reverseDistanceShortDelta1}\,(\ref{itm_reverseDistanceShortDelta1_2}), a contradiction to the choice of~$R$.
So let us assume that $x$ is the starting point of~$P_2$.
Let us suppose that (\ref{itm_exponentialDivergence1}) is false.
Since $d(P_2^x(R+\delta),Q)>\delta$, there is a point $u$ on~$xP_1a$ with $d(u,P_2^x(R+\delta))\leq\delta$.
This must lie on $P_1^x(R)P_1a$ as~$P_2$ is a geodesic.
Now let $v$ be on~$xP_2P_2^x(R+\delta)$ such that there is a point $v'$ on $P_1^x(R)P_1a$ with $d(v',v)\leq\delta$ but for no point $w$ on $xP_2v$ with $d(w,v)\geq\varepsilon$ there exists a point $w'$ on $P_1^x(R)P_1a$ with $d(w',w)\leq\delta$.
Note that we also have $d(v,Q)>\delta$ since (\ref{itm_exponentialDivergence1}) is false.
Let $w$ be on $xP_2v$ with $d(w,v)=\varepsilon$ if it exists and $w=x$ otherwise.
Then there is a point $w'$ on~$P_1$ with $d(w',w)\leq\delta$ since (\ref{itm_exponentialDivergence1}) is false.
This point must lie on $xP_1P_1^x(R)$.
Considering a geodesic triangle with end points $w', v, v'$ where the side between $w'$ and~$v'$ is $w'P_1 v'$ and the other two sides are directed towards~$v$, we obtain $d(w',v')\leq (2\delta+\varepsilon+1)f(\delta+1)$ by Proposition~\ref{prop_reverseDistanceShortDelta1}\,(\ref{itm_reverseDistanceShortDelta1_1}).
Since $P_1^x(R)$ lies on the side between $w'$ and~$v'$, this is a contradiction to the choice of~$R$.
This finishes the proof of (\ref{itm_exponentialDivergence1}) and~(\ref{itm_exponentialDivergence2}).

We consider $Q$ as being indexed with the empty word.
We define a set of directed paths $Q_\sigma$ with starting and end points on~$P$ and points $u_\sigma$ on~$P$ indexed by finite words over $\{0,1\}$ such that the following holds.
\begin{enumerate}[\rm (I)]
\item If $Q_\sigma$ is a directed path of length more than one, then $u_i$ is a point on~$P$ such that $u_i$ halves the length of the directed subpath of~$P$ between the starting and the end points of~$Q_\sigma$.
\item $Q_{\sigma 0}$ is a geodesic with the same starting point as $Q_\sigma$ and with end point~$u_\sigma$.
\item $Q_{\sigma 1}$ is a geodesic with $u_\sigma$ as starting point and the same end point as~$Q_\sigma$.
\end{enumerate}
It follows from the choice of the points $u_\sigma$ that the length of any $\{0,1\}$-word used as index is at most $\ln(\ell(P))$.
By applying Lemma~\ref{lem_parallelSideCloseToAndFrom} recursively, we obtain for every $n\geq 0$ a point $v_n$ on some $Q_\sigma$, where $\sigma$ has length~$n$, such that $d(x,v_n)\leq R+2\delta+nk$ if $x$ is the starting point of~$P_1$ or $d(v_n,x)\leq R+\delta+nk$ if $x$ is the end point of~$P_1$.
Note that for $n=\ln(\ell(P))$ the point $v_n$ has distance at most~$1$ to and from~$P$.
So we found a point $y$ on~$P$ with either $d(x,y)\leq R+2\delta+nk+1$ or $d(y,x)\leq R+2\delta+nk+1$.
Since $P$ lies outside of $\cB^+_{R+r}(x)\cap \cB^-_{R+r}(x)$, we have
\[
R+r\leq R+2\delta+nk+1\leq R+2\delta+k\ln(\ell(P))+1.
\]
Thus, we have
\[
e(r)<2^{\frac{r-2\delta-1}{k}}\leq\ell(P),
\]
which shows the assertion.
\end{proof}

As mentioned earlier, the reverse implication of Proposition~\ref{prop_exponentialDivergence} does not hold as the following example shows.

\begin{ex}\label{ex_GeodDivNoDelta}
Let $D$ be the digraph whose vertex set is given by two directed rays $x_0x_1\ldots$ and $y_1y_2\ldots$ with additional edges  $x_0y_1$ and $x_iy_i$ for all $i\in\nat$.
Then there is no $\delta\geq 0$ such that all geodesic triangles with end vertices $x_0,x_i,y_i$ are $\delta$-thin.
Furthermore, every function $e\colon\R\to\real$ with $e(r)>2$ for all $r\in\R$ is a divergence function for~$D$.
\end{ex}

A property for geodesic spaces is that divergence of geodesics that is faster than linear already implies that it is exponential.
For geodesic \hm\ spaces, we pose it as an open problem.

\begin{prob}
Do there exist geodesic \hm\ spaces in which geodesics diverge exponentially but not superlinearly?
\end{prob}

\section{Geodesic stability}\label{sec_GeoStab}

Another property in geodesic spaces that is equivalent to hyperbolicity is \emph{geodesic stability}, that is, for every two points, all $(\gamma,c)$-\qg s between those two points lie close to each other.
This is the main step in showing that hyperbolicity is preserved under \qis.

A geodesic \hm\ space $X$ satisfies \emph{geodesic stability} if for all $\gamma\geq 1$ and $c\geq 0$ there exists a $\kappa\geq 0$ such that, for all $x,y\in X$ and all $(\gamma,c)$-\qg s $P$ and~$Q$ from~$x$ to~$y$, every point of~$P$ lies in $\cB^+_\kappa(Q)\cap \cB^-_\kappa(Q)$.

As a first step, we prove that geodesics lie close to and from \qg s with the same starting and end points (Proposition~\ref{prop_GeodStab1}) and then that \qg s lie close to and from geodesics with the same starting and end points (Proposition~\ref{prop_GeodStab2}).

\begin{prop}\label{prop_GeodStab1}
Let $X$ be a $\delta$-hyperbolic geodesic \hm\ space that satisfies (\ref{itm_Bounded1}) and (\ref{itm_Bounded2}) for the function $f\colon\R\to\R$.
Let $\gamma\geq 1$ and $c\geq 0$.
Then there is a constant $\kappa=\kappa(\delta,\gamma,c,f)$ depending only on $\delta,\gamma,c$ and~$f$ such that for all $x,y\in X$ every $x$-$y$ geodesic lies in the $\kappa$-out-ball and in the $\kappa$-in-ball of every $(\gamma,c)$-\qg\ from~$x$ to~$y$.
\end{prop}

\begin{proof}
By Proposition~\ref{prop_exponentialDivergence}, we find a superlinear divergence function $e\colon\R\to\R$ for~$X$.\footnote{Proposition~\ref{prop_exponentialDivergence} implies that $e$ is exponential, but in this proof it suffices to know that $e$ is superlinear.}
First, we will show the following.
\begin{txteq}\label{itm_GeodStab1_1}
There exists $\kappa_1$ such that for all $x,y\in X$ and for all $x$-$y$ geodesics $P$ and $(\gamma,c)$-\qg s $Q$ from~$x$ to~$y$ we have $d^\lrarrow(z,Q)\leq \kappa_1$ for all $z$ on~$P$.
\end{txteq}

Let us suppose that~(\ref{itm_GeodStab1_1}) does not hold.
Then there are $x,y\in X$, an $x$-$y$ geodesic~$P$ and a $(\gamma,c)$-\qg\ $Q$ from~$x$ to~$y$ such that there exists $z$ on~$P$ with $2d^\lrarrow(z,Q)> e(0)f(\delta+1)+f(\delta)$.
Let $\varepsilon>0$.
We choose $z$ on~$P$ such that for no $z'$ on~$P$ we have $d^\lrarrow(z',Q)\geq D:=d^\lrarrow(z,Q)+\varepsilon$.
Set
\[
K:=(D+\delta)f(\delta+1)+\delta+\varepsilon
\]
and
\[
L:=(K+D)(2+f(\delta+1))+\varepsilon.
\]

In order to show (\ref{itm_GeodStab1_1}) we will first show the following
\begin{txteq}\label{itm_GeodStab1_2}
There are points $a^+,a^-$ on $xPz$ and $a_Q$ on~$Q$ with $D-\varepsilon\leq d(a^+,z)\leq L$ and $D-\varepsilon\leq d(a^-,z)\leq L$, with $d(a^+,a_Q)\leq K$ and $d(a_Q,a^-)\leq K$ and such that all $a^+$-$a_Q$ and $a_Q$-$a^-$ geodesics lie outside of $B:=\cB^+_D(z)\cup \cB^-_D(z)$.
\end{txteq}
This obviously holds if $d(x,z)\leq L$ by choosing $a^+=a^-=a_Q=x$.
Otherwise, let $a$ on $xPz$ with $d(a,z)=K+D+(K+D)f(\delta+1)+\varepsilon$.
By the choice of~$z$, we have $d^\lrarrow(a,Q)\leq D$.

First, let us consider the case that $d(Q,a)< D$.
Let $a_Q$ be on~$Q$ with $d(Q,a)\leq d(a_Q,a)+\varepsilon$.
We consider a geodesic triangle with end points $x,a,a_Q$ such that the side $P'$ between $x$ and~$a$ is a directed subpath of~$P$.
Let $Q_1$ be the side from $x$ to~$a_Q$ and let $Q_2$ be the side from~$a_Q$ to~$a$.
Then there are $u,v$ on~$Q_1$ with $d(u,v)\leq\varepsilon$ such that $u$ lies in $\cB^+_\delta(P')$ and $v$ lies in $\cB^-_\delta(Q_2)$.
Let $w$ be on~$P'$ with $d(w,u)\leq\delta$ and let $w'$ be on~$Q_2$ with $d(v,w')\leq\delta$.
Applying Proposition~\ref{prop_reverseDistanceShortDelta1} to the geodesic triangle with end vertices $v,a_Q,w'$, we obtain $d(v,a_Q)\leq (D+\delta) f(\delta+1)$ and thus
\[
d(w,a_Q)\leq d(w,u)+d(u,v)+d(v,a_Q)\leq \delta+\varepsilon+ (D+\delta)f(\delta+1)=K.
\]
It follows that
\[
d(w,a)\leq d(w,a_Q)+d(a_Q,a)\leq K+D.
\]
We set $a^+:=w$ and $a^-:=a$.
Then we have
\[
d(a^+,z)=d(a^+,a^-)+d(a^-,z)\leq (K+D)(2+f(\delta+1))+\varepsilon=L
\]
and $d(a^+,a_Q)=d(w,a_Q)\leq K$.
If there were a point on an $a^+$-$a_Q$ or $a_Q$-$a^-$ geodesic inside~$B$, then we apply Proposition~\ref{prop_reverseDistanceShortDelta1} and obtain $d(a^+,z)\leq (K+D)f(\delta+1)$ or $d(a^-,z)\leq 2Df(\delta+1)$, which is impossible.

If $d(a,Q)\leq D$, then let $a_Q$ be on~$Q$ with $d(a,Q)\leq d(a,a_Q)+\varepsilon$.
Similarly as in the first case, we use the geodesic triangle with end points $a,a_Q,y$ to find a point $w$ on $aPy$ with $d(a_Q,w)\leq K$ and $d(a,w)\leq K+D$.
We set $a^+:=a$ and $a^-:=w$.
Again by Proposition~\ref{prop_reverseDistanceShortDelta1}, the $a^+$-$a_Q$ and $a_Q$-$a^-$ geodesics lie outside of~$B$.
Thus, we have proved~(\ref{itm_GeodStab1_2}).

A completely symmetric argument shows the following.
\begin{txteq}\label{itm_GeodStab1_3}
There are points $b^+,b^-$ on $zPy$ and $b_Q$ on~$Q$ with $D-\varepsilon\leq d(z,b^+)\leq L$ and $D-\varepsilon\leq d(z,b^-)\leq L$, with $d(b^+,b_Q)\leq K$ and $d(b_Q,b^-)\leq K$ and such that all $b^+$-$b_Q$ and $b_Q$-$b^-$ geodesics lie outside of~$B$.
\end{txteq}

Let $z_a$ be on $xPz$ and let $z_b$ be on $zPy$ with $d(z_a,z)=D=d(z,z_b)$.
Then we have $d(z_a,z_b)=2D>e(0)$ and
\[
d(z_b,z_a)\geq \frac{d(z_a,z_b)-f(\delta)}{f(\delta+1)}>e(0)
\]
by Proposition~\ref{prop_reverseDistanceShortDelta1}.
If $a_Q$ lies on $xQb_Q$, then let $R$ be the composition of an $a^+$-$a_Q$ geodesic with $a_QQb_Q$ and with a $b_Q$-$b^-$ geodesic.
Then $R$ is a directed $a^+$-$b^-$ path outside of~$B$.
The composition of an $a_Q$-$a^-$ geodesic with $a^-Pb^+$ and a $b^+$-$b_Q$ geodesic shows that $d(a_Q,b_Q)\leq 2(K+L)$.
So we have
\[
\ell(R)\leq 2K+2\gamma(K+L)+c,
\]
which is linear in~$D$.

Let us now assume that $b_Q$ lies on $xQa_Q$.
Let $R$ be the composition of a $b^+$-$b_Q$ geodesic with $b_QQa_Q$ and with an $a_Q$-$a^-$ geodesic.
Then $R$ is a $b^+$-$a^-$ path outside of~$B$.
The composition of an $a_Q$-$a^-$ geodesic with $a^-Pb^+$ and a $b^+$-$b_Q$ geodesic shows $d(a_Q,b_Q)\leq 2(K+L)$.
Since $b_Q$ lies on the $x$-$a_Q$ subpath of~$Q$, there is a $b_Q$-$a_Q$ geodesic, which has length at most $2(K+L)f(\delta+1)+f(\delta)$ by Proposition~\ref{prop_reverseDistanceShortDelta1}.
Thus, we have
\[
\ell(R)\leq 2K+2\gamma(K+L)f(\delta+1)+\gamma f(\delta)+c,
\]
which is also linear in~$D$.

So in each case we obtain a contradiction to~$e$ being a superlinear divergence function.
We thus proved~(\ref{itm_GeodStab1_1}).

Let
\[
\kappa_2>2(\kappa_1+7\delta+(2\delta+\kappa_1)f(\delta+1))
\]
and set
\begin{align*}
\lambda_1:=&((\kappa_2+\kappa_1)f(\delta+1)+\kappa_1)f(\delta+1),\\
\lambda_2:=&\gamma\lambda_1+c,\\
\kappa_3:=&\lambda_2+7\delta+2\delta f(\delta+1)+\kappa_2/2.
\end{align*}
We will prove that the assertion holds for~$\kappa=\kappa_2+\kappa_3$.

Let $x,y\in X$, let $P$ be an $x$-$y$ geodesic and let~$Q$ be a $(\gamma,c)$-\qg\ from~$x$ to~$y$.
For all $0\leq i$ with $i\kappa_2\leq d(x,y)$ let $x_i$ be the point on~$P$ with $d(x,x_i)=i\kappa_2$ and $y_i$ be a point on~$Q$ with $d^\lrarrow(x_i,y_i)\leq\kappa_1$.
We assume $x_0=y_0$.
We will show the following.
\begin{txteq}\label{itm_GeodStab1_4}
For all $i\in\N$ such that $i\kappa_2\leq d(x,y)$, we have $d(x_i,Q)\leq\kappa_3$ and $d(Q,x_i)\leq\kappa_3$.
\end{txteq}
Note that the assertion follows for the constant~$\kappa$ from~(\ref{itm_GeodStab1_4}) directly.

Let $N\in\N$ be maximum such that $N\kappa_2\leq d(x,y)$ and set $x_{N+1}:=y_{N+1}:=y$.
For all $i\in\{0,\ldots,N-1\}$ such that $y_i$ lies on $xQy_{i+1}$, there is either a directed $y_i$-$x_{i+1}$ path or a directed $x_i$-$y_{i+1}$ path depending on the direction of the shortest path between $x_i$ and~$y_i$.
Applying Proposition~\ref{prop_reverseDistanceShortDelta1} twice to geodesic triangles with end points either $x_i,y_i,x_{i+1}$ and $y_i,x_{i+1},y_{i+1}$ or $x_i,x_{i+1},y_{i+1}$ and $x_i,y_i,y_{i+1}$ we obtain $d(y_i,y_{i+1})\leq \lambda_1$.

Let $n\in\{1,\ldots, N\}$ such that $d(x_n,y_n)\leq\kappa_2$ and let $i\in\{0,\ldots,n-1\}$ be largest such that $y_i$ lies on $xQy_n$.
Then $y_i$ lies on $xQy_{i+1}$ and $y_n$ lies on $y_iQy_{i+1}$.
Since $Q$ is a $(\gamma,c)$-\qg, the length of the path $y_iQy_{i+1}$ is at most $\lambda_2$ and thus we have $d(y_i,y_n)\leq\lambda_2$.

We will show $d(y_i,x_n)\leq\kappa_3$.
This holds trivially, if $d(y_i,x_i)\leq\kappa_1$.
So let us assume that $d(x_i,y_i)\leq\kappa_1$.
Note that there is a directed $x_i$-$y_n$ path, so we find an $x_i$-$y_n$ geodesic~$R_1$.
Let $P'$ the subpath of~$P$ between $x_i$ and~$x_n$ and let $Q'$ be a $y_i$-$y_n$ geodesic.
Let $R_2$ be a $x_n$-$y_n$ geodesic and let $R_3$ be a geodesic between $x_i$ and~$y_i$.
We consider two geodesic triangles: one with $x_i,x_n,y_n$ as its end points and $P',R_1,R_2$ as its sides and the other with $x_i,y_i,y_n$ as its end points and $Q',R_1,R_3$ as its sides.
Let $u$ be on~$P'$ with $d(u,x_n)=\kappa_2/2$.
By hyperbolicity, there is either a point $v$ on~$R_1$ with $d(v,u)\leq\delta$ or a point~$v'$ on~$R_2$ with $d(u,v')\leq\delta$.
The latter case leads to a contradiction since then Proposition~\ref{prop_reverseDistanceShortDelta1} implies $d(u,x_n)\leq (\delta+\kappa_1)f(\delta+1)<\kappa_2/2$.
So we find a point $v$ on~$R_1$ with $d(v,u)\leq\delta$.

Applying Lemma~\ref{lem_parallelSideCloseToAndFrom}, we find a point $w$ on either $Q'$ or~$R_3$ with $d(w,v)\leq 6\delta+2\delta f(\delta+1)$.
If $w$ is on~$R_3$, we find a path from~$x_i$ via $w$ and~$v$ to~$u$ of length at most $\kappa_1+7\delta+2\delta f(\delta+1)$, which contradicts $d(x_i,u)=\kappa_2/2$.
Thus, $w$ lies on~$Q'$ and hence, the composition of $y_iQ'w$ with a $w$-$v$ geodesic, a $v$-$u$ geodesic and $uPx_n$ shows
\[
d(y_i,x_n)\leq \lambda_2+7\delta+2\delta f(\delta+1)+\kappa_2/2=\kappa_3.
\]

Now let us consider the case that for $n\in\{1,\ldots,N\}$ we have $d(y_n,x_n)\leq\kappa_1$.
We follow a completely symmetric argument as in the previous case with reversed directions.
Most importantly, in this case $x_i$ is chosen so that $i\in\{n+1,\ldots,N+1\}$ is smallest such that $y_i$ lies on $y_nQy$.

This finishes the proof of~(\ref{itm_GeodStab1_4}) and thus the proof of our assertion as mentioned earlier.
\end{proof}

\begin{prop}\label{prop_GeodStab2}
Let $X$ be a $\delta$-hyperbolic geodesic \hm\ space satisfying (\ref{itm_Bounded1}) and (\ref{itm_Bounded2}) for the function $f\colon\R\to\R$.
Let $\gamma\geq 1$ and $c\geq 0$.
Then there is a constant $\lambda=\lambda(\delta,\gamma,c,f)$ depending only on $\delta,\gamma,c$ and~$f$ such that for all $x,y\in X$ every $(\gamma,c)$-\qg\ from~$x$ to~$y$ lies in the $\lambda$-out-ball and in the $\lambda$-in-ball of every $x$-$y$ geodesic.
\end{prop}

\begin{proof}
Let $\kappa$ be the constant from Proposition~\ref{prop_GeodStab1}.
Let $x,y\in X$ and let $P$ be an $x$-$y$ geodesic and $Q$ be a $(\gamma,c)$-\qg\ from~$x$ to~$y$.
Let $z$ be on~$Q$ such that either $d(z,P)>\kappa$ or $d(P,z)>\kappa$.
Set $Q_1:=xQz$ and $Q_2:=zQy$.
Let $\varepsilon>0$ and let $u$ be on~$P$ such that there exists $w$ on~$Q_1$ with $d(u,w)\leq\kappa$ and such that for no point $u'$ on~$uPy$ with $d(u,u')\geq\varepsilon$ there is $w'$ on~$Q_1$ with $d(u',w')\leq\kappa$.

If there is $w'$ on~$Q_2$ with $d(w',u)\leq\kappa$, then $d(w',w)\leq 2\kappa$ and since there is a directed $w$-$w'$ path, we have $d(w,w')\leq 2\kappa f(\delta+1) + f(\delta)$ by Proposition~\ref{prop_reverseDistanceShortDelta1}.
So the distance from~$w$ to~$w'$ on~$Q$ is at most $2\gamma\kappa f(\delta+1)+\gamma f(\delta)+c$ and hence $z$ and~$u$ lie on a closed directed walk of length at most $2\kappa+2\gamma\kappa f(\delta+1)+\gamma f(\delta)+c$.

Let us now assume that there is no $w'$ on~$Q_2$ with $d(w',u)\leq\kappa$.
In particular, we have $u\neq y$.
By Proposition~\ref{prop_GeodStab1}, there is a point $w'$ on~$Q_1$ with $d(w',u)\leq\kappa$.
Let $v$ be on~$uPy$ with $d(u,v)=\varepsilon$.
By the choice of~$u$, there is a point $v'$ on~$Q_2$ with $d(v,v')\leq\kappa$.
So we have $d(w',v')\leq 2\kappa+\varepsilon$ and hence the length of the directed $w'$-$v'$ subpath of~$Q$ is at most $\gamma (2\kappa+\varepsilon)+c$.
Note that this directed subpath contains~$z$.
We consider the geodesic triangle with end points $v,v',y$ such that the side $P'$ between $v$ and~$y$ is $vPy$, the side $R_1$ is a $v'$-$y$ geodesic and the side $R_2$ is a $v$-$v'$ geodesic.
Then there $a,b$ on~$R_1$ with $a$ on~$v'R_1b$ and $d(a,b)\leq\varepsilon$ such that $d(R_2,a)\leq\delta$ and $d(b,P')\leq\delta$.
By Proposition~\ref{prop_reverseDistanceShortDelta1}, we conclude $d(v',a)\leq (\kappa+\delta)f(\delta+1)$.
So we have
\[
d(v',P)\leq d(v',a)+d(a,b)+d(b,P)\leq (\kappa+\delta)f(\delta+1)+\varepsilon+\delta
\]
and hence
\[
d(z,P)\leq d(z,v')+d(v',P)\leq (\kappa+\delta)f(\delta+1)+\varepsilon+\delta+\gamma (2\kappa+\varepsilon)+c.
\]
With a symmetric argument by considering a geodesic triangle with end vertices $x,w',u$, we conclude
\[
d(P,z)\leq (\kappa+\delta)f(\delta+1)+\varepsilon+\delta+ \gamma (2\kappa+\varepsilon)+c.
\]
So the assertion follows for
\[
\lambda:=\max\{2\kappa+2\gamma\kappa f(\delta+1)+\gamma f(\delta)+c,(\kappa+\delta)f(\delta+1)+1+\delta+\gamma (2\kappa+1)+c\}.\qedhere
\]
\end{proof}

We obtain the following corollary directly from Propositions~\ref{prop_GeodStab1} and~\ref{prop_GeodStab2} the following.

\begin{cor}\label{cor_Delta1GeodStab}
Every hyperbolic geodesic \hm\ space that satisfies (\ref{itm_Bounded1}) and (\ref{itm_Bounded2}) satisfies geodesic stability.\qed
\end{cor}

To compare geodesic stability with divergence of geodesics, we note that exponential divergence of geodesics does not imply geodesic stability as Example~\ref{ex_GeodDivNoDelta} shows.
The reverse is also false as our next example shows.
That example is a digraph that satisfies geodesic stability but whose geodesics do not diverge properly and that is not hyperbolic.

\begin{ex}
Let $x_0,x_1\ldots$ and $y_0,y_1,\ldots$ be two disjoint sequences of distinct vertices.
For every $i$, we add an edge $x_iy_i$, a path from $x_{i+1}$ to~$x_i$ of length $i+1$ and a path from $y_i$ to~$y_{i+1}$ of length $i+1$ such that all these new paths are pairwise disjoint and meet the two sequences only in its end vertices.
For $n\in\N$, we take two geodesics: one that ends at~$x_0$ and passed through all $x_i$ for $i\leq n$ and one that starts at~$x_0$ and passes through all $y_i$ for $i\leq n$.
There are vertices on the first geodesic that have distance at least $n/2$ to the other geodesic, but the edge $x_n y_n$ shows that these two geodesics do not diverge properly.
The geodesic triangles with end vertices $x_n$, $x_{n+1}$ and $y_{n+1}$ are not $(n/2-1)$-thin.
So $D$ is not hyperbolic.

Let $R_1$ be the geodesic anti-ray that passes through all $x_i$ and ends at~$x_0$ and let $R_2$ be the geodesic ray that starts at~$y_0$ and passes through all~$y_i$.
Then every directed path that is not a subpath of $R_1$ or~$R_2$ can be obtained by composing a subpath of~$R_1$ with a path consisting of an edge $x_iy_i$ and with a subpath of~$R_2$.
Hence, if a geodesic and a $(\gamma,c)$-\qg\ have the same first and the same last vertex, then they differ in that an edge $x_iy_i$ of the geodesic is replaced by $P=x_iR_1x_jy_jR_2y_i$ for some $j<i$.
Note that we have $\ell(P)\leq \gamma+c$ as $d(x_i,y_i)=1$.
If follows that the distance of points on~$P$ from and to the geodesic is at most $\ell(P)$ and hence bounded by $\gamma+c$.
Trivially, all points on the geodesic lie within $1$ from and to the \qg.
It follows that the digraph satisfies geodesic stability.
\end{ex}

\section{Quasi-isometries}\label{sec_QI}

In this section, we look at \qis\ of geodesic \hm\ spaces and ask which of our properties are preserved by \qis.
This is motivated by the fact that hyperbolicity of geodesic metric spaces is preserved by \qis, see e.\,g.\ \cite[Theorem III.H.1.9]{BridsonHaefliger}.
We will prove that hyperbolicity (with the additional assumptions of (\ref{itm_Bounded1}) and (\ref{itm_Bounded2})) and geodesic stability are preserved under \qis\ and pose it as open problem what happens for (superlinear/exponential) divergence of geodesics.

\begin{prop}
Geodesic stability is preserved by \qis.
\end{prop}

\begin{proof}
Let $X$ and $Y$ be geodesic \hm\ spaces such that $Y$ satisfies geodesic stability and let $f\colon X\to Y$ be a $(\lambda,c)$-\qiy\ for some $\gamma\geq 1$ and $c\geq 0$.
Let $x_1,x_2\in X$, let $P$ be an $x_1$-$x_2$ geodesic and let $Q$ be a $(\lambda,\ell)$-\qg\ from $x_1$ to~$x_2$ for some $\lambda\geq 1$ and $\ell\geq 0$.
Then there are constants $\gamma'\geq 1$ and $c'\geq 0$ such that the images of~$P$ and~$Q$ under~$f$ are $(\gamma',c')$-\qg s, see Gray and Kambites \cite[Proposition 1]{GK-SemimetricSpaces}.
By geodesic stability, there is a constant $\kappa\geq 0$ such that
\[
f(P)\sub \cB^+_\kappa(f(Q))\cap\cB^-_\kappa(f(Q))
\]
and
\[
f(Q)\sub \cB^+_\kappa(f(P))\cap\cB^-_\kappa(f(P)).
\]
Since $f$ is a \qiy, there is a constant $\kappa'\geq 0$ such that
\[
P\sub \cB^+_{\kappa'}(Q)\cap\cB^-_{\kappa'}(Q)
\]
and
\[
Q\sub \cB^+_{\kappa'}(P)\cap\cB^-_{\kappa'}(P).\qedhere
\]
\end{proof}

The proof that hyperbolicity is preserved under \qis\ relies also on their geodesic stability.

\begin{prop}\label{prop_QIPreserveHyp}
Let $X$ and $Y$ be two geodesic \hm\ spaces such that $X$ is hyperbolic and satisfies (\ref{itm_Bounded1}) and~(\ref{itm_Bounded2}).
If $X$ is \qi\ to~$Y$, then $Y$ is hyperbolic.
\end{prop}

\begin{proof}
Let $X$ be $\delta$-hyperbolic for some $\delta\geq 0$.
Let $f\colon Y\to X$ be a $(\gamma,c)$-\qiy\ for some $\gamma\geq 1$ and $c\geq 0$.
We consider a geodesic triangle in~$Y$ with end points $x,y,z$ and sides $P_1,P_2,P_3$, where $P_1$ is an $x$-$y$ geodesic, $P_2$ a geodesic between $y$ and~$z$ and $P_3$ are geodesic between $x$ and~$z$.
Then the images of the sides are $(\gamma,c)$-\qg s in~$X$.
Let $\kappa$ be a constant such that geodesic stability holds for~$\kappa$ in~$X$ with respect to $(\gamma,c)$-\qg s, see Corollary~\ref{cor_Delta1GeodStab}.

Let $u$ be on~$P_1$.
By geodesic stability, we find $u_1,u_2$ on some $f(x)$-$f(y)$ geodesic in~$X$ with $d(u,u_1)\leq\kappa$ and $d(u_2,u)\leq\kappa$.
As geodesic triangles are $\delta$-thin, there are $v_1,v_2$ on geodesics either from $f(x)$ to~$f(z)$ or from $f(y)$ to~$f(z)$ with $d(u_1,v_1)\leq\delta$ and $d(v_2,u_2)\leq\delta$.
Applying geodesic stability once more, there are $w_1,w_2$ on the images under~$f$ of $P_3$ and $P_2$ with $d(v_1,w_1)\leq\kappa$ and $d(w_2,v_2)\leq\kappa$.
So we have $d(u,w_1)\leq 2\kappa+\delta$ and $d(w_2,u)\leq 2\kappa+\delta$.
Since $f$ is a \qiy, it follows that there exists $\delta'\geq 0$, depending only on $\kappa,\delta,\gamma$ and~$c$ such that the geodesic triangle is $\delta'$-thin.
\end{proof}

We end this section by a stating the corresponding problem for divergence of geodesics.

\begin{prob}
Is proper divergence (or superlinar/exponential divergence) of geode\-sics preserved under \qis?
\end{prob}

\section{Hyperbolic semigroups}\label{sec_semigroups}

Let $S$ be a semigroups and let $A$ be a finite generating set of~$S$.
The \emph{(right) Cayley digraph} of~$S$ (with respect to~$A$) has $S$ as its vertex set and edges from $x$ to $xa$ for all $x\in S$ and $a\in A$.
This way, $S$ is a \hm\ space.
A straight-forward argument shows that different finite generating sets define \qi\ Cayley digraphs, see Gray and Kambites \cite[Proposition 4]{GK-SemimetricSpaces}.

A finitely generated semigroup is \emph{hyperbolic} if it has a finite generating set such that its Cayley digraph with respect to that generating set is hyperbolic.

Gray and Kambites \cite{GK-HyperbolicDigraphsMonoids} asked whether every Cayley digraph of a hyperbolic monoid with respect to a finite generating set is hyperbolic.
Proposition~\ref{prop_QIPreserveHyp} allows us to give at least a partial answer.

\begin{prop}
For every hyperbolic right cancellative semigroup, each of its Cayley digraphs with respect to finite generating sets is hyperbolic.
\end{prop}

\begin{proof}
Let $S$ be a hyperbolic right cancellative semigroup.
Let $A$ be a finite generating set and let $D$ be the Cayley digraph of~$S$ with respect to~$A$.
Then the out-degree of all vertices of~$D$ is~$|A|$.
Similarly, since $S$ is right cancellative, the in-degree of all vertices is bounded by~$|A|$.
In particular, $D$ satisfies (\ref{itm_Bounded1}) and~(\ref{itm_Bounded2}).

Let $B$ be a finite generating set of~$S$ such that the Cayley digraph $D'$ of~$S$ with respect to~$B$ is hyperbolic.
We also have the in- and out-degrees of all vertices bounded by $|B|$ in this situation.
Thus, Lemma~\ref{lem_bddDegImpliesB1B2} implies that $D'$ satisfies (\ref{itm_Bounded1}) and~(\ref{itm_Bounded2}), too.
Changing finite generating sets of semigroups lead to quasi-isometric Cayley digraphs, see e.\,g.\ \cite[Proposition 4]{GK-SemimetricSpaces}.
So Proposition~\ref{prop_QIPreserveHyp} directly implies that $D$ is hyperbolic.
\end{proof}

\providecommand{\bysame}{\leavevmode\hbox to3em{\hrulefill}\thinspace}
\providecommand{\MR}{\relax\ifhmode\unskip\space\fi MR }
\providecommand{\MRhref}[2]{%
  \href{http://www.ams.org/mathscinet-getitem?mr=#1}{#2}
}
\providecommand{\href}[2]{#2}

\end{document}